\documentclass[a4paper, 12pt, reqno]{amsart}

\pdfoutput=1

\usepackage[utf8]{inputenc}
\usepackage[margin=2.25cm, headsep=0.625cm, footskip=0.625cm]{geometry}
\usepackage{amsmath, amsthm, amssymb, mathtools, physics}
\usepackage[hang, flushmargin]{footmisc}
\usepackage[english]{isodate}
\usepackage{enumitem}
\usepackage{xcolor}
\usepackage{tikz}
\usepackage[colorlinks=true, linkcolor=blue, linkbordercolor=blue, citecolor=red, citebordercolor=red, urlcolor=indigo, linktocpage=true]{hyperref}
\usepackage[capitalise, nameinlink, noabbrev, nosort]{cleveref}
\usepackage{doi}

\usetikzlibrary{decorations.markings}

\definecolor{indigo}{HTML}{492DA5}

\allowdisplaybreaks

\setenumerate{listparindent=\parindent}

\makeatletter\g@addto@macro\bfseries{\boldmath}\makeatother

\makeatletter
\let\origsection\section
\renewcommand\section{\@ifstar{\starsection}{\nostarsection}}
\newcommand\sectionspace{\vspace{0.5ex}}
\newcommand\nostarsection[1]{\sectionspace\origsection{#1}\sectionspace}
\newcommand\starsection[1]{\sectionspace\origsection*{#1}\sectionspace}
\makeatother

\setlist[enumerate]{font=\normalfont}
\crefname{enumi}{}{}

\crefname{case}{case}{cases}
\creflabelformat{case}{#2#1#3}

\crefname{page}{page}{pages}

\numberwithin{equation}{section}
\crefname{equation}{equation}{equations}

\crefname{inequality}{inequality}{inequalities}
\creflabelformat{inequality}{#2(#1)#3}

\crefname{condition}{condition}{conditions}
\creflabelformat{condition}{#2(#1)#3}

\newtheorem{theorem}{Theorem}[section]

\newtheorem{thm}[theorem]{Theorem}
\crefname{thm}{Theorem}{Theorems}

\newtheorem{lemma}[theorem]{Lemma}
\crefname{lemma}{Lemma}{Lemmas}

\newtheorem{prop}[theorem]{Proposition}
\crefname{prop}{Proposition}{Propositions}

\newtheorem{cor}[theorem]{Corollary}
\crefname{cor}{Corollary}{Corollaries}

\theoremstyle{definition}

\crefname{defn}{Definition}{Definitions}

\newtheorem{remark}[theorem]{Remark}
\crefname{remark}{Remark}{Remarks}

\newtheorem{example}[theorem]{Example}
\crefname{example}{Example}{Examples}

\newcommand{\hl}[1]{\textcolor{magenta}{\emph{#1}}}
\newcommand{\case}[1]{\textbf{\Cref{#1}:}}
\newcommand{\C}{\mathbb{C}}
\newcommand{\F}{\mathbb{F}}
\newcommand{\N}{\mathbb{N}}

\newcommand{\GG}{\mathcal{G}}
\newcommand{\HH}{\mathcal{H}}
\newcommand{\OO}{\mathcal{O}}
\newcommand{\PG}{\mathcal{P}_\GG}
\newcommand{\GGo}{\GG^{(0)}}
\newcommand{\GGc}{\GG^{(2)}}
\newcommand{\Bco}[1]{B^\mathrm{co}(#1)}
\newcommand{\BcoG}{\Bco{\GG}}
\newcommand{\supp}{\operatorname{supp}}
\newcommand{\vecspan}{\operatorname{span}}
\newcommand{\Iso}{\operatorname{Iso}}
\newcommand{\restr}[1]{\ensuremath{\vert_{#1}}}
\newcommand{\maxnorm}[1]{\ensuremath{\norm{#1}_{\mathrm{max}}}}
\newcommand{\rnorm}[1]{\ensuremath{\norm{#1}_r}}
\DeclarePairedDelimiter{\ceil}{\lceil}{\rceil}

\hypersetup{pdfauthor={Becky Armstrong, Lisa Orloff Clark, Mahya Ghandehari, Eun Ji Kang, and Dilian Yang}, pdftitle={Representing topological full groups in Steinberg algebras and C*-algebras}}

\date{\today}
\title[Representing topological full groups in Steinberg algebras \& C*-algebras]{Representing topological full groups in Steinberg algebras and C*-algebras}
\author[Armstrong]{Becky Armstrong}
\author[Clark]{Lisa Orloff Clark}
\author[Ghandehari]{Mahya Ghandehari}
\author[Kang]{Eun Ji Kang}
\author[Yang]{Dilian Yang}

\address[B.~Armstrong and L.O.~Clark]{School of Mathematics and Statistics, Victoria University of Wellington, PO Box 600, Wellington 6140, NEW ZEALAND}
\email[B.~Armstrong]{\href{mailto:becky.armstrong@vuw.ac.nz}{becky.armstrong@vuw.ac.nz}}
\email[L.O.~Clark]{\href{mailto:lisa.orloffclark@vuw.ac.nz}{lisa.orloffclark@vuw.ac.nz}}

\address[M.~Ghandehari]{Department of Mathematical Sciences, University of Delaware, New\-ark, DE, 19716, USA}
\email{\href{mailto:mahya@udel.edu}{mahya@udel.edu}}

\address[E.J.~Kang]{Research Institute of Mathematics, Seoul National University, Seoul 08826, KOREA}
\email{\href{mailto:kkang3333@gmail.com}{kkang3333@gmail.com}}

\address[D.~Yang]{Department of Mathematics and Statistics, University of Windsor, Windsor, ON N9B 3P4, CANADA}
\email{\href{mailto:dyang@uwindsor.ca}{dyang@uwindsor.ca}}


\thanks{This research collaboration began as part of the project-oriented workshop ``Women in Operator Algebras~II'' (21w5199), funded by the Banff International Research Station. The research was also supported by the Marsden Fund of the Royal Society of New Zealand (grant number 21-VUW-156); the Deutsche Forschungsgemeinschaft (DFG, German Research Foundation) under Germany's Excellence Strategy -- EXC 2044 -- 390685587, Mathematics M\"unster -- Dynamics -- Geometry -- Structure; the Deutsche Forschungsgemeinschaft (DFG, German Research Foundation) -- Project-ID 427320536 – SFB 1442; ERC Advanced Grant 834267 -- AMAREC; NSF Grant DMS-1902301; NSERC Discovery Grants 808235 and 823065; RS-2023-00238961; and NRF-2022M3H3A1098237. The authors would like to thank Owen Tanner for helpful suggestions on a draft of this paper; Eduardo Scarparo for insightful comments on the first preprint; and the anonymous referee for their careful reading, quick response, and useful feedback.}

\begin{document}

\begin{abstract}
We study the natural representation of the topological full group of an ample Hausdorff groupoid in the groupoid's complex Steinberg algebra and in its full and reduced C*-algebras. We characterise precisely when this representation is injective and show that it is rarely surjective. We then restrict our attention to discrete groupoids, which provide unexpected insight into the behaviour of the representation of the topological full group in the full and reduced groupoid C*-algebras. We show that the image of the representation is not dense in the full groupoid C*-algebra unless the groupoid is a group, and we provide an example showing that the image of the representation may still be dense in the reduced groupoid C*-algebra even when the groupoid is not a group.
\end{abstract}

\maketitle

\section{Introduction}
Topological full groups of ample Hausdorff groupoids were introduced by Matui \cite{Matui2012} as a generalisation of the topological full groups studied by Giordano, Putnam, and Skau in the context of Cantor minimal systems \cite{GPS1999}. Matui showed in \cite[Theorem~3.10]{Matui2015} that for any two minimal effective Hausdorff \'etale groupoids whose unit spaces are Cantor sets, the groupoids are isomorphic if and only if their topological full groups are isomorphic. This is equivalent to there being a diagonal-preserving isomorphism of the Steinberg algebras of the groupoids; see \cite[Theorem~3.1]{ABHS2017}. It is therefore clear that there are strong connections between the topological full groups and Steinberg algebras of ample Hausdorff groupoids.

In addition to being a groupoid invariant, topological full groups have enticing connections to some infamous open questions. For example, they give presentations of Thompson’s groups \cite{LV2020, MM2017, Matui2015, Yang2022}, and have already been used to solve several important problems in group theory; see \cite{BHM2022, JM2013, JNdlS2016, Nek2018, SWZ2019}. Recent results also reveal interesting connections between topological full groups and the elusive simplicity problem for group C*-algebras; see \cite{BS2019, LBMB2018, Scarparo2023}. It is this latter problem that motivates our study.

For every ample Hausdorff groupoid $\GG$ with compact unit space, there are natural representations of the topological full group of $\GG$ in the complex Steinberg algebra of $\GG$ and in the full and reduced C*-algebras of $\GG$. It is known that these representations often fail to be injective. We make this statement precise by showing that the representation of the topological full group taking values in the Steinberg algebra of the groupoid is almost never injective. In particular, we show in \cref{thm: main} that injectivity fails when
\begin{enumerate}[label=(\arabic*)]
\item the groupoid is all isotropy and has at least $2$ nontrivial isotropy groups; or
\item the groupoid is not all isotropy and has at least $3$ non-unit elements.
\end{enumerate}
We then show that this representation is almost never surjective as a map into the complex Steinberg algebra. In fact, we show in \cref{cor: surj iff group} that the representation is surjective onto the Steinberg algebra if and only if $\GG$ is a group. However, strangely, the image of the representation of the topological full group may still be dense in the full or reduced groupoid C*-algebras. For example, density of the image holds for the representation of the topological full group associated to the Cuntz groupoid (that is, the boundary-path groupoid of the directed graph with a single vertex and two edges) into the Cuntz algebra $\OO_2$; see \cite[Remark~4.7]{BS2019} and \cite[Proposition~5.3]{HO2017}. \Cref{eg: F_2 sqcup F_2} provides another such example.

Our proof techniques for the results in \cref{sec: injectivity,sec: surjectivity} were developed by first considering these questions for discrete groupoids. The arguments in our proof of \cref{thm: main} in particular are quite combinatorial in nature.

In \cref{sec: discrete groupoids} we demonstrate that surprising things can happen in the setting of discrete groupoids. In \cref{thm: discrete full iff group}, we show that the image of the representation of the topological full group of a discrete groupoid with finite unit space is dense in the full groupoid C*\nobreakdash-algebra if and only if the groupoid is a group. (Note that the Cuntz groupoid mentioned above is not discrete, and thus this result does not hold for ample Hausdorff groupoids in general; see \cref{rem: Cuntz groupoid}.) In \cref{eg: F_2 sqcup F_2} we demonstrate that it is possible for the image of the representation of the topological full group of a discrete groupoid to be dense in the reduced groupoid C*-algebra even when the groupoid is not a group. Finally, in \cref{cor: isomorphism} we combine our results from \cref{sec: injectivity,sec: surjectivity,sec: discrete groupoids} to show that the representation of the topological full group of an ample Hausdorff groupoid $\GG$ with compact unit space is an isomorphism into the Steinberg algebra of $\GG$ if and only if $\GG$ is a group, and that when $\GG$ is discrete with finite unit space, the same result holds for the extension of this representation to the full C*-algebra.

\section{Preliminaries}

\subsection{Groupoids}
A \hl{groupoid} $\GG$ is a small category in which every morphism $\gamma \in \GG$ has a unique inverse $\gamma^{-1} \in \GG$. Throughout, we assume that all groupoids are nonempty. We define the \hl{range} and \hl{source} of each $\gamma \in \GG$ by $r(\gamma) \coloneqq \gamma \gamma^{-1}$ and $s(\gamma) \coloneqq \gamma^{-1} \gamma$, respectively, where composition is read from right to left. We write
\[
\GGc = \{ (\alpha,\beta) \in \GG \times \GG \mid s(\alpha) = r(\beta) \}
\]
for the set of \hl{composable pairs} in $\GG$, and we write $\GGo = r(\GG) = s(\GG)$ for the \hl{unit space} of $\GG$. Note that a groupoid $\GG$ is a group if and only if $\GGo$ is a singleton. A \hl{topological groupoid} is a groupoid endowed with a topology under which composition and inversion are continuous. A \hl{Hausdorff groupoid} is a topological groupoid with a locally compact Hausdorff topology. If $\GG$ is a Hausdorff groupoid, then $\GGo$ is closed in $\GG$. A topological groupoid $\GG$ is \hl{\'etale} if the range and source maps $r,s\colon \GG \to \GGo$ are local homeomorphisms. A subset $B \subseteq \GG$ is called a \hl{bisection} of $\GG$ if $r\restr{B}$ and $s\restr{B}$ are injective. If $B$ is an open bisection of an \'etale groupoid $\GG$, then $r\restr{B}$ and $s\restr{B}$ are homeomorphisms onto open subsets of $\GGo$. Every \'etale groupoid has a basis consisting of open bisections; see \cite[Proposition~3.5]{Exel2008}. We say that an \'etale groupoid is \hl{ample} if it has a basis of \emph{compact} open bisections. By \cite[Proposition~4.1]{Exel2010}, a Hausdorff \'etale groupoid is ample if and only if its unit space is totally disconnected. If $\GG$ is an \'etale groupoid, then the unit space $\GGo$ is open in $\GG$, and for all $u, v \in \GGo$, each of the sets
\[
\GG^u \coloneqq r^{-1}(u), \ \GG_v \coloneqq s^{-1}(v), \text{ and } \ \GG_v^u \coloneqq \GG^u \cap \GG_v
\]
is discrete with respect to the relative topology induced by $\GG$. The \hl{isotropy group} of a unit $u \in \GGo$ is the group
\[
\GG^u_u = \{\gamma \in \GG \mid r(\gamma) = s(\gamma) = u\},
\]
and the \hl{isotropy subgroupoid} of $\GG$ is the collection
\[
\Iso(\GG) \coloneqq \bigcup_{u \in \GGo} \, \GG^u_u = \{ \gamma \in \GG \mid r(\gamma) = s(\gamma) \}.
\]

Let $\GG$ be a Hausdorff \'etale groupoid. For each continuous function $f\colon \GG \to \C$, we define $\supp(f) \coloneqq \overline{\{ \gamma \in \GG : f(\gamma) \ne 0 \}}$. We write $C_c(\GG)$ for the collection of continuous compactly supported complex-valued functions on $\GG$. This is a $*$-algebra with respect to the convolution product
\[
(f * g)(\gamma) = \sum_{\alpha\beta = \gamma} f(\alpha) g(\beta)
\]
and $*$-involution $f^*(\gamma) = \overline{f(\gamma^{-1})}$ for $f,g \in C_c(\GG)$ and $\gamma \in \GG$. Given a Hilbert space $\HH$, we write $B(\HH)$ for the C*-algebra of bounded linear operators on $\HH$. The \hl{full groupoid C*-algebra} $C^*(\GG)$ is the completion of $C_c(\GG)$ with respect to the \hl{full C*-norm}
\[
\maxnorm{f} \coloneqq \sup\{ \norm{\pi(f)} \mid \pi\colon C_c(\GG) \to B(\HH) \text{ is a $*$-representation for some } \HH \}.
\]
For each $u \in \GGo$, there is a $*$-representation $\pi_u\colon C_c(\GG) \to B(\ell^2(\GG_u))$, called the \hl{regular representation} of $C_c(\GG)$ associated to $u$, such that
\[
\pi_u(f) \delta_\gamma = \sum_{\alpha \in \GG_{r(\gamma)}} f(\alpha) \delta_{\alpha\gamma} \quad \text{for } f \in C_c(\GG) \text{ and } \gamma \in \GG_u.
\]
The \hl{reduced groupoid C*-algebra} $C_r^*(\GG)$ is the completion of $C_c(\GG)$ with respect to the \hl{reduced C*-norm}
\[
\rnorm{f} \coloneqq \sup\{ \norm{\pi_u(f)} \mid u \in \GGo \}.
\]
See \cite[Chapter~II]{Renault1980} or \cite[Chapter~9]{Sims2020} for details.

The \hl{(complex) Steinberg algebra} of an ample Hausdorff groupoid $\GG$ is the collection
\begin{align*}
A(\GG) \coloneqq&~\vecspan\{ 1_U\colon \GG \to \C \mid U \text{ is a compact open bisection of } \GG \} \\
=&~\{ f \in C_c(\GG) \mid f \text{ is locally constant} \}
\end{align*}
equipped with the convolution product and $*$-involution defined above. If $\GG$ is discrete, then $A(\GG) = C_c(\GG)$. In general, $A(\GG)$ is dense in $C_c(\GG)$ with respect to both the full and reduced C*-norms (see \cite[Proposition~4.2]{CFST2014}), and for all $f \in A(\GG)$, we have
\begin{equation} \label{eqn: Steinberg *-reps are bounded}
\maxnorm{f} \coloneqq \sup\{ \norm{\pi(f)} \mid \pi\colon A(\GG) \to B(\HH) \text{ is a $*$-representation for some } \HH \}
\end{equation}
(see \cite[Theorem~7.1]{CZ2022}). Note that a discrete group $G$ may be viewed as an ample Hausdorff groupoid, and in this case the singletons in $G$ are all compact open bisections, and so the Steinberg algebra $A(G)$ is just the complex group ring $\C G$, which is generated by the point-mass functions $\delta_g \coloneqq 1_{\{g\}}$ for $g \in G$. See \cite{CFST2014, Steinberg2010} for further details on Steinberg algebras.

\subsection{Topological full groups}
Let $\GG$ be an ample groupoid with compact unit space $\GGo$. We write $\BcoG$ for the inverse semigroup of compact open bisections of $\GG$, and we say that a bisection $B$ of $\GG$ is \hl{full} if $r(B) = s(B) = \GGo$. We define the \hl{topological full group} of $\GG$ to be the (discrete) group
\[
F(\GG) \coloneqq \{ B \in \BcoG \mid B \text{ is full} \}
\]
equipped with the operations
\[
AB \coloneqq \{ \alpha\beta \mid (\alpha, \beta) \in (A \times B) \cap \GGc \} \ \text{ and } \ B^{-1} \coloneqq \{ \gamma^{-1} \mid \gamma \in B \}
\]
for all $A, B\in F(\GG)$. Note that if $\GG$ is a discrete group, then
\[
F(\GG) = \BcoG = \{ \{g\} \mid g \in \GG \} \cong \GG.
\]
See \cite{Matui2017, Nek2019} for further details on topological full groups.

Let $\GG$ be an ample Hausdorff groupoid with compact unit space. Given compact open bisections $U$ and $V$ of $\GG$, we have
\[
1_U * 1_V = 1_{UV} \quad \text{ and } \quad (1_U)^* = 1_{U^{-1}}.
\]
It follows that there is a $*$-homomorphism $\pi\colon \C F(\GG) \to A(\GG)$ satisfying $\pi(\delta_U) = 1_U$, which we call the \hl{representation} of $F(\GG)$ in $A(\GG)$. This representation is studied extensively in \cite{BS2019}, as are the C*-completions $\overline{\pi(\C F(\GG))}^{\maxnorm{\cdot}}$ and $\overline{\pi(\C F(\GG))}^{\rnorm{\cdot}}$ of its image. In this paper we investigate the necessary and sufficient conditions under which $\pi$ is injective and surjective.

\begin{remark}
In \cite[Definition~3.2]{NO2019} Nyland and Ortega define the topological full group of an (effective) ample Hausdorff groupoid with a unit space that is not necessarily compact. Since the Steinberg algebra of an ample Hausdorff groupoid $\GG$ is unital (with unit $1_{\GGo}$) if and only if the unit space $\GGo$ is compact, it is impossible to represent the topological full group of $\GG$ in $A(\GG)$ (or in $C^*(\GG)$ or $C_r^*(\GG)$) unless $\GGo$ is compact. It is for this reason that we restrict our attention in this paper to ample Hausdorff groupoids with compact unit space.
\end{remark}

\section{Lack of injectivity for ample Hausdorff groupoids}
\label{sec: injectivity}

In this section we characterise precisely when the representation $\pi\colon \delta_U \mapsto 1_U$ of $\C F(\GG)$ in $A(\GG)$ is injective. In particular, we show in \cref{thm: main} that $\pi$ is injective if and only if either $\GG$ consists entirely of isotropy and has at most one nontrivial isotropy group, or $\GG$ contains exactly $2$ non-unit elements outside its isotropy.

\begin{prop} \label{prop: rep not inj}
Let $\GG$ be an ample Hausdorff groupoid with compact unit space $\GGo$. Suppose that either
\begin{enumerate}[label = (\arabic*), ref=\arabic*]
\item \label[condition]{cond: iso rep not inj} $\GG = \Iso(\GG)$ and $\GG$ has at least two nontrivial isotropy groups; that is, there exist $u,v \in \GGo$ such that $u \ne v$ and $\abs{\GG_u^u}, \abs{\GG_v^v} > 1$; or
\item \label[condition]{cond: non-iso rep not inj} $\GG \ne \Iso(\GG)$ and
$\abs{\GG \setminus \GGo} \ge 3$.
\end{enumerate}
Then the representation $\pi\colon \C F(\GG) \to A(\GG)$ is not injective.
\end{prop}

\begin{proof}
We first assume that \cref{cond: iso rep not inj} holds. Fix $\gamma_1, \gamma_2 \in \GG \setminus \GGo$ such that $r(\gamma_1) \ne r(\gamma_2)$. Since $\GGo$ is Hausdorff and $\GG$ is all isotropy, we can find disjoint compact open bisections $B_1$ and $B_2$ containing $\gamma_1$ and $\gamma_2$, respectively, such that
\[
r(B_1) = s(B_1), \quad r(B_2) = s(B_2), \quad \text{and } \quad r(B_1) \cap r(B_2) = \varnothing.
\]
Set $R \coloneqq \GGo \setminus \big(r(B_1) \cup r(B_2)\big)$, and note that $R$ is a compact open bisection of $\GG$. Consider the following disjoint unions:
\begin{align*}
U_1 &\coloneqq B_1 \cup r(B_2) \cup R, \\
U_2 &\coloneqq B_2 \cup r(B_1) \cup R, \text{ and} \\
U_3 &\coloneqq B_1 \cup B_2 \cup R.
\end{align*}
It is straightforward to verify that $U_1$, $U_2$, and $U_3$ are distinct elements of $F(\GG)$. Define $a \coloneqq \delta_{U_1} + \delta_{U_2} - \delta_{U_3} - \delta_{\GGo}$. Then
\[
\pi(a) = 1_{U_1} + 1_{U_2} - 1_{U_3} - 1_{\GGo} = 1_{r(B_2)} + 1_{r(B_1)} + 1_R - 1_{\GGo} = 0,
\]
and so $0 \ne a \in \ker{\pi}$. Hence $\pi$ is not injective.

We now assume that \cref{cond: non-iso rep not inj} holds instead. Then there exist
\begin{equation} \label[condition]{cond: not inj}
\text{$\gamma_1 \in \GG \setminus \GGo$ and $\gamma_2 \in \GG \setminus \Iso(\GG)$ such that $\gamma_1 \ne \gamma_2$ and $\gamma_1 \ne \gamma_2^{-1}$.}
\end{equation}
For any $\gamma_1, \gamma_2$ satisfying \cref{cond: not inj}, we have $r(\gamma_2) \ne s(\gamma_2)$, and either $\gamma_1 \notin \Iso(\GG)$ or $\gamma_1 \in \Iso(\GG)$. By replacing $\gamma_1$ with $\gamma_1^{-1}$ or $\gamma_2$ with $\gamma_2^{-1}$ if necessary, we can summarise all possible cases as follows:
\begin{enumerate}[label = (\roman*)]
\item \label[case]{case: joined arrows} $\gamma_1 \notin \Iso(\GG)$ and $s(\gamma_1) = r(\gamma_2)$ and $s(\gamma_2) \ne r(\gamma_1)$;
\item \label[case]{case: separate arrows} $\gamma_1 \notin \Iso(\GG)$ and $r(\gamma_1)$, $s(\gamma_1)$, $r(\gamma_2)$, and $s(\gamma_2)$ are all distinct;
\item \label[case]{case: separate loop and arrow} $\gamma_1 \in \Iso(\GG)$ and $s(\gamma_1)$, $r(\gamma_2)$, and $s(\gamma_2)$ are all distinct;
\item \label[case]{case: loop with tail} $\gamma_1 \in \Iso(\GG)$ and $s(\gamma_1) = r(\gamma_2)$; and
\item \label[case]{case: parallel arrows} $\gamma_1 \notin \Iso(\GG)$ and $s(\gamma_1) = r(\gamma_2)$ and $s(\gamma_2) = r(\gamma_1)$.
\end{enumerate}

\begin{figure}[h]
\begin{tikzpicture}[font=\footnotesize]
\coordinate (A) at (0,0);
\coordinate (B) at (2,0);
\coordinate (C) at (4,0);
\coordinate (D) at (1,0.3);
\coordinate (E) at (3,0.3);
\coordinate (F) at (3,-0.3);
\node[circle, draw, inner sep = 1pt, fill = black] at (A) {};
\node[circle, draw, inner sep = 1pt, fill = black] at (B) {};
\node[circle, draw, inner sep = 1pt, fill = black] at (C) {};
\draw[postaction = {decorate, decoration = {markings, mark = at position 0.5 with {\arrow{>}}}}, bend right] (B) to (A);
\node[above] at (D) {$\gamma_1$};
\draw[postaction = {decorate, decoration = {markings, mark = at position 0.5 with {\arrow{>}}}}, bend right] (C) to (B);
\node[above] at (E) {$\gamma_2$};
\draw[postaction = {decorate, decoration = {markings, mark = at position 0.5 with {\arrow{>}}}}, bend right] (B) to (C);
\node[below] at (F) {$\gamma_2^{-1}$};
\coordinate (G) at (7.2,0);
\coordinate (H) at (9.2,0);
\coordinate (I) at (6.5,0);
\coordinate (J) at (8.2,0.3);
\coordinate (K) at (8.2,-0.3);
\node[circle, draw, inner sep = 1pt, fill = black] at (G) {};
\node[circle, draw, inner sep = 1pt, fill = black] at (H) {};
\draw[postaction = {decorate, decoration = {markings, mark = at position 0.5 with {\arrow{<}}}}, loop above, out = -130, in = 130, looseness = 8, min distance = 15mm] (G) to (G);
\node[left] at (I) {$\gamma_1$};
\draw[postaction = {decorate, decoration = {markings, mark = at position 0.5 with {\arrow{>}}}}, bend right] (H) to (G);
\node[above] at (J) {$\gamma_2$};
\draw[postaction = {decorate, decoration = {markings, mark = at position 0.5 with {\arrow{>}}}}, bend right] (G) to (H);
\node[below] at (K) {$\gamma_2^{-1}$};
\coordinate (L) at (11.2,0);
\coordinate (M) at (13.2,0);
\coordinate (N) at (12.2,0.7);
\coordinate (O) at (12.2,0.3);
\coordinate (P) at (12.2,-0.25);
\node[circle, draw, inner sep = 1pt, fill = black] at (L) {};
\node[circle, draw, inner sep = 1pt, fill = black] at (M) {};
\draw[postaction = {decorate, decoration = {markings, mark = at position 0.5 with {\arrow{>}}}}, bend left = 80] (L) to (M);
\node[above] at (N) {$\gamma_1$};
\draw[postaction = {decorate, decoration = {markings, mark = at position 0.5 with {\arrow{>}}}}, bend right] (M) to (L);
\node[below] at (O) {$\gamma_2$};
\draw[postaction = {decorate, decoration = {markings, mark = at position 0.5 with {\arrow{>}}}}, bend right] (L) to (M);
\node[below] at (P) {$\gamma_2^{-1}$};
\end{tikzpicture}
\vspace{-0.5ex}
\caption{From left to right: \cref{case: joined arrows,case: loop with tail,case: parallel arrows}.}
\end{figure}

\vspace{1ex}

\noindent Moreover, we can reduce \cref{case: parallel arrows} to \cref{case: loop with tail} by replacing $\gamma_1$ with $\gamma_2 \gamma_1$. Therefore, it suffices to show that $\ker{\pi}$ is nontrivial in each of the \crefrange{case: joined arrows}{case: loop with tail}.

\case{case: joined arrows} Suppose that the hypotheses of \cref{case: joined arrows} hold, and let $B_1$ and $B_2$ be compact open bisections containing $\gamma_1$ and $\gamma_2$, respectively. Since $\GGo$ is Hausdorff and since $r(\gamma_1)$, $s(\gamma_1)$, and $s(\gamma_2)$ are all distinct, we may assume that $r(B_1)$, $s(B_1)$, and $s(B_2)$ are mutually disjoint by shrinking $B_1$ and $B_2$ if necessary. Moreover, since $s(\gamma_1) = r(\gamma_2)$, we can replace $B_1$ with $B_1 \big(s(B_1) \cap r(B_2)\big)$ and $B_2$ with $\big(s(B_1) \cap r(B_2)\big) B_2$, and thus without loss of generality we may assume that $s(B_1) = r(B_2)$. Set $R \coloneqq \GGo \setminus \big(r(B_1) \cup s(B_1) \cup s(B_2)\big)$, and note that $R$ is a compact open bisection of $\GG$. Consider the following disjoint unions that are distinct elements of $F(\GG)$:
\begin{align*}
U \coloneqq&~B_1 \cup B_2 \cup (B_1 B_2)^{-1} \cup R, \\
U^{-1} =&~B_1^{-1} \cup B_2^{-1} \cup (B_1 B_2) \cup R, \\
U_1 \coloneqq&~B_1 \cup B_1^{-1} \cup s(B_2) \cup R, \\
U_2 \coloneqq&~B_2 \cup B_2^{-1} \cup r(B_1) \cup R, \text{ and} \\
U_3 \coloneqq&~B_1 B_2 \cup (B_1 B_2)^{-1} \cup s(B_1) \cup R.
\end{align*}
Define $a \coloneqq \delta_U + \delta_{U^{-1}} - \delta_{U_1} - \delta_{U_2} - \delta_{U_3} + \delta_{\GGo}$. Then
\begin{align*}
\pi(a) = 1_U + 1_{U^{-1}} - 1_{U_1} - 1_{U_2} - 1_{U_3} + 1_{\GGo}
= -1_{s(B_2)} - 1_{r(B_1)} - 1_{s(B_1)} - 1_R + 1_{\GGo} = 0,
\end{align*}
and so $a \in \ker{\pi} {\setminus} \{0\}$. Hence $\pi$ is not injective.

\case{case: separate arrows} Now suppose that the hypotheses of \cref{case: separate arrows} hold, and let $B_1$ and $B_2$ be compact open bisections containing $\gamma_1$ and $\gamma_2$, respectively. Since $\GGo$ is Hausdorff and since $r(\gamma_1)$, $s(\gamma_1)$, $r(\gamma_2)$, and $s(\gamma_2)$ are all distinct, we may assume that $r(B_1)$, $s(B_1)$, $r(B_2)$, and $s(B_2)$ are mutually disjoint by shrinking $B_1$ and $B_2$ if necessary. Set
\[
R \coloneqq \GGo \setminus \big(r(B_1) \cup s(B_1) \cup r(B_2) \cup s(B_2)\big),
\]
and note that $R$ is a compact open bisection of $\GG$. Consider the following disjoint unions that are distinct elements of $F(\GG)$:
\begin{align*}
U_1 &\coloneqq B_1 \cup B_1^{-1} \cup r(B_2) \cup s(B_2) \cup R, \\
U_2 &\coloneqq B_2 \cup B_2^{-1} \cup r(B_1) \cup s(B_1) \cup R, \text{ and} \\
U_3 &\coloneqq B_1 \cup B_1^{-1} \cup B_2 \cup B_2^{-1} \cup R.
\end{align*}
It is straightforward to verify that $a \coloneqq \delta_{U_1} + \delta_{U_2} - \delta_{U_3} - \delta_{\GGo} \in \ker{\pi} {\setminus} \{0\}$, and hence $\pi$ is not injective.

\case{case: separate loop and arrow} Next, suppose that the hypotheses of \cref{case: separate loop and arrow} hold, and let $B'_1$ and $B_2$ be compact open bisections containing $\gamma_1$ and $\gamma_2$, respectively. Since $\GGo$ is Hausdorff and since $s(\gamma_1)$, $r(\gamma_2)$, and $s(\gamma_2)$ are all distinct, we may assume that $s(B'_1)$, $r(B_2)$, and $s(B_2)$ are mutually disjoint by shrinking $B'_1$ and $B_2$ if necessary. Let $V \coloneqq r(B'_1) \cap s(B'_1)$, and define $B_1 \coloneqq V B'_1 V$. Then $B_1$ is a compact open bisection containing $\gamma_1$, because $\GG$ is an ample Hausdorff groupoid and $r(\gamma_1) = s(\gamma_1) \in V$. Suppose that $r(B_1) \ne s(B_1)$. Then there exists $\alpha \in B_1$ such that $r(\alpha) \notin s(B_1)$ or $s(\alpha) \notin r(B_1)$. In either case, $\alpha \notin \Iso(\GG)$ and $r(\alpha), s(\alpha) \in V \subseteq s(B'_1)$. Thus, since $\gamma_2 \in \GG \setminus \Iso(\GG)$ and $r(\gamma_2), s(\gamma_2) \in \GGo {\setminus} s(B'_1)$, we deduce that $\alpha \ne \gamma_2$ and $\alpha \ne \gamma_2^{-1}$, and that $r(\alpha)$, $s(\alpha)$, $r(\gamma_2)$, and $s(\gamma_2)$ are all distinct. So if $r(B_1) \ne s(B_1)$, then \cref{case: separate loop and arrow} can be reduced to \cref{case: separate arrows} by replacing $\gamma_1$ with $\alpha$. Now suppose that $r(B_1) = s(B_1)$. Since $r(B_1) \subseteq V \subseteq s(B'_1)$, we know that $r(B_1)$, $r(B_2)$, and $s(B_2)$ are mutually disjoint. Set $R \coloneqq \GGo \setminus \big(r(B_1) \cup r(B_2) \cup s(B_2)\big)$, and note that $R$ is a compact open bisection of $\GG$. Consider the following disjoint unions that are distinct elements of $F(\GG)$:
\begin{align*}
U_1 &\coloneqq r(B_1) \cup B_2 \cup B_2^{-1} \cup R, \\
U_2 &\coloneqq B_1 \cup r(B_2) \cup s(B_2) \cup R, \text{ and} \\
U_3 &\coloneqq B_1 \cup B_2 \cup B_2^{-1} \cup R.
\end{align*}
It is straightforward to verify that $a \coloneqq \delta_{U_1} + \delta_{U_2} - \delta_{U_3} - \delta_{\GGo} \in \ker{\pi} {\setminus} \{0\}$, and hence $\pi$ is not injective in this case either.

\case{case: loop with tail} Finally, suppose that the hypotheses of \cref{case: loop with tail} hold, and let $B'_1$ and $B'_2$ be compact open bisections containing $\gamma_1$ and $\gamma_2$, respectively. Since $\GGo$ is Hausdorff and since $r(\gamma_2) \ne s(\gamma_2)$, we may assume that $r(B'_2) \cap s(B'_2) = \varnothing$. Let $W \coloneqq r(B'_1) \cap s(B'_1) \cap r(B'_2)$, and define  $B_1 \coloneqq W B'_1 W$ and $B_2 \coloneqq s(B_1) B'_2$. Then $B_1$ and $B_2$ are compact open bisections containing $\gamma_1$ and $\gamma_2$, respectively, because $\GG$ is an ample Hausdorff groupoid and $r(\gamma_1) = s(\gamma_1) = r(\gamma_2) \in W$. Since $s(B_1) \subseteq W \subseteq r(B'_2)$, we have $r(B_2) = s(B_1) \cap r(B'_2) = s(B_1)$. Suppose that $r(B_1) \ne s(B_1)$. Then there exists $\alpha \in B_1$ such that $r(\alpha) \notin s(B_1)$ or $s(\alpha) \notin r(B_1)$. In either case, $\alpha \notin \Iso(\GG)$ and $r(\alpha), s(\alpha) \in W \subseteq r(B'_2)$, so there exists $\beta \in B'_2$ such that $r(\beta) = s(\alpha) \in W$. Since $s(\beta) \in s(B'_2) \subseteq \GGo {\setminus} r(B'_2) \subseteq \GGo {\setminus}W$, we know that $s(\beta) \ne r(\alpha)$, $s(\beta) \ne s(\alpha)$, and $s(\beta) \ne r(\beta)$. Thus $\beta \in \GG \setminus \Iso(\GG)$, $\alpha \ne \beta$, and $\alpha \ne \beta^{-1}$. So if $r(B_1) \ne s(B_1)$, then \cref{case: loop with tail} can be reduced to \cref{case: joined arrows} by replacing $\gamma_1$ and $\gamma_2$ with $\alpha$ and $\beta$, respectively. Now suppose that $r(B_1) = s(B_1) = r(B_2)$. Since $r(B_2) \subseteq r(B'_2)$ and $s(B_2) \subseteq s(B'_2)$, we know that $r(B_2) \cap s(B_2) = \varnothing$. Set $R \coloneqq \GGo \setminus \big(r(B_2) \cup s(B_2)\big)$, and note that $R$ is a compact open bisection of $\GG$. Consider the following disjoint unions that are elements of $F(\GG)$:
\begin{align*}
U_1 &\coloneqq B_2 \cup (B_1 B_2)^{-1} \cup R, \\
U_2 &\coloneqq B_1 B_2 \cup B_2^{-1} \cup R, \\
U_3 &\coloneqq B_2 \cup B_2^{-1} \cup R, \text{ and} \\
U_4 &\coloneqq B_1 B_2 \cup (B_1 B_2)^{-1} \cup R.
\end{align*}
To see that $U_1$, $U_2$, $U_3$, and $U_4$ are distinct elements of $F(\GG)$, note that since $\gamma_1 \notin \GGo$, we have $\gamma_2 \ne \gamma_1 \gamma_2$, and hence $\gamma_2 \notin B_1 B_2$, because $\gamma_1 \gamma_2$ is the unique element of the bisection $B_1 B_2$ with source $s(\gamma_2)$. It is straightforward to verify that $a \coloneqq \delta_{U_1} + \delta_{U_2} - \delta_{U_3} - \delta_{U_4} \in \ker{\pi} {\setminus} \{0\}$, and hence $\pi$ is not injective.
\end{proof}

We conclude this section by proving that the converse of \cref{prop: rep not inj} also holds.

\begin{thm} \label{thm: main}
Let $\GG$ be an ample Hausdorff groupoid with compact unit space. The representation $\pi\colon \C F(\GG) \to A(\GG)$ is injective if and only if
\begin{enumerate}[label=(\arabic*), ref=\arabic*]
\item \label[condition]{cond: all isotropy} $\GG = \Iso(\GG)$ and $\GG$ has at most one nontrivial isotropy group; or
\item \label[condition]{cond: not all isotropy} $\GG \ne \Iso(\GG)$ and $\abs{\GG \setminus \GGo} < 3$.
\end{enumerate}
\end{thm}

\begin{proof}
If $\pi$ is injective, then the result follows by the contrapositive of \cref{prop: rep not inj}. For the converse, first suppose that \cref{cond: all isotropy} holds. If $\GG = \GGo$, then $F(\GG) = \{\GGo\}$, and so $\C F(\GG) \cong \C$, and hence $\pi$ is injective. Suppose that $\GG \ne \GGo$. Then there exists a nontrivial discrete group $\Gamma$ with identity $e_\Gamma$ such that $\GG = \Gamma \sqcup X$, where $X = \GGo {\setminus} \{e_\Gamma\}$. Since $\GG$ is Hausdorff and \'etale, $X = (\GG {\setminus} \{e_\Gamma\}) \cap \GGo$ is open in $\GG$. We claim that $X$ is compact. To see this, first observe that since $\GG$ is Hausdorff, $\GGo$ is closed, and so $\Gamma {\setminus} \{e_\Gamma\} = \GG {\setminus} \GGo$ is open in $\GG$. Thus, since $\GG$ is \'etale, $\{e_\Gamma\} = r(\Gamma {\setminus} \{e_\Gamma\})$ is open in $\GG$, and so $X = (\GG {\setminus} \{e_\Gamma\}) \cap \GGo$ is closed. Now, since $X \subseteq \GGo$ and $\GGo$ is compact by hypothesis, $X$ must also be compact, as claimed. For each $\gamma \in \Gamma$, choose a compact open bisection $U_\gamma$ of $\GG$ containing $\gamma$. Then $U_\gamma \cap \Gamma = \{\gamma\}$. Since $X = \GGo {\setminus} \{e_\Gamma\}$ is compact and open in $\GG$, we have $V_\gamma \coloneqq U_\gamma \cup X = \{\gamma\} \sqcup X \in F(\GG)$, and it follows that $F(\GG) = \big\{ \{\gamma\} \sqcup X \mid \gamma \in \Gamma \big\}$. Now, let $f \in \ker{\pi} \subseteq \C F(\GG)$. Then for some $m \in \N$, there exist $c_1, \dotsc, c_m \in \C$ and $\gamma_1, \dotsc, \gamma_m \in \Gamma$ such that $\gamma_i \ne \gamma_j$ whenever $i \ne j$, and $f = \sum_{i=1}^m c_i \, \delta_{\{\gamma_i\} \sqcup X}$. Since $\pi(f) = 0$, we have
\[
c_k = \Big( \sum_{i=1}^m c_i 1_{\{\gamma_i\}} + \big( \sum_{i=1}^m c_i \big) 1_X \Big)(\gamma_k) = \pi(f)(\gamma_k) = 0
\]
for each $k \in \{1, \dotsc, m\}$, and so $f = 0$. Thus $\pi$ is injective.

Now suppose that \cref{cond: not all isotropy} holds. Since $\GG \ne \Iso(\GG)$, there exists $\gamma \in \GG \setminus \Iso(\GG)$, and it follows that $\gamma$ and $\gamma^{-1}$ are distinct elements of $\GG \setminus \GGo$. Thus $\abs{\GG \setminus \GGo} = 2$, and so $\GG = \GGo \sqcup \{\gamma, \gamma^{-1}\}$. In particular, $\GG$ is compact. Since $\GG$ is Hausdorff, $\GGo$ is closed, and so $\{\gamma, \gamma^{-1}\} = \GG \setminus \GGo$ is open. Thus $\{r(\gamma), s(\gamma)\} = r(\{\gamma, \gamma^{-1}\})$ is open since $\GG$ is \'etale. Therefore, $U \coloneqq \GG \setminus \{r(\gamma), s(\gamma)\}$ is a closed subset of $\GG$, and by the compactness of $\GG$ it follows that $U$ is a full compact open bisection containing $\gamma$ and $\gamma^{-1}$. In fact, given $V \in F(\GG)$ with $\gamma \in V$, we must have $r(\gamma), s(\gamma) \notin V$, and so $V = U$. It follows that $F(\GG) = \{U, \GGo\}$. Suppose that $f = a \delta_U + b \delta_{\GGo} \in \ker(\pi)$ for some $a, b \in \C$. Then $a = \pi(f)(\gamma) = 0$ and $b = \pi(f)(r(\gamma)) = 0$, and so $f = 0$. Thus $\pi$ is injective.
\end{proof}

\section{Lack of surjectivity for ample Hausdorff groupoids}
\label{sec: surjectivity}

In this section we study the image of the representation $\pi\colon \delta_U \mapsto 1_U$ of $\C F(\GG)$ in $A(\GG)$. In particular, we show in \cref{cor: surj iff group} that $\pi$ is surjective if and only if $\GG$ is a group.

We begin by proving certain properties for elements of the image of $\pi$. Recall (for instance, from \cite[Section~2.2]{Matui2016}) that there are linear maps $r_*, s_*\colon A(\GG) \to A(\GGo)$ given by
\[
r_*f(u) \coloneqq \sum_{\gamma \in \GG^u} f(\gamma) \quad \text{and} \quad s_*f(u) \coloneqq \sum_{\gamma \in \GG_u} f(\gamma), \ \text{ for all } f \in A(\GG) \text{ and } u \in \GGo;
\]
and there is a linear map $\delta_1\colon A(\GG) \to A(\GGo)$ given by $\delta_1 \coloneqq s_* - r_*$.

\begin{prop} \label{prop: ample rep image}
Let $\GG$ be an ample Hausdorff groupoid with compact unit space $\GGo$. Then
\begin{enumerate}[label=(\alph*)]
\item \label{item: ample rep image in ker delta_1} $\pi(\C F(\GG)) \subseteq \{ f \in A(\GG) \mid r_*f(u) = s_*f(v) \text{ for all } u, v \in \GGo \} \subseteq \ker{\delta_1}$;\, and
\item \label{item: r_* and s_* of ample rep image} $r_*\big(\pi(\C F(\GG))\big) = \C 1_{\GGo} = s_*\big(\pi(\C F(\GG))\big)$.
\end{enumerate}
\end{prop}

\begin{proof}
For part~\cref{item: ample rep image in ker delta_1}, fix $f \in \pi(\C F(\GG))$. Then there exist $U_1, \dotsc, U_m \in F(\GG)$ and $c_1, \dotsc, c_m \in \C$ such that
\[
f = \pi\Big( \sum_{i=1}^m c_i \, \delta_{U_i} \Big) = \sum_{i=1}^m c_i 1_{U_i}.
\]
Fix $u, v \in \GGo$. For each $i \in \{1, \dotsc, m\}$, the sets $U_i \cap \GG^u$ and $U_i \cap \GG_v$ are singletons because $U_i$ is a full bisection of $\GG$. Thus
\[
r_*f(u) \,=\, \sum_{\gamma \in \GG^u} f(\gamma) \,=\, \sum_{\gamma \in \GG^u} \ \sum_{i : \gamma \in U_i} c_i \,=\, \sum_{i=1}^m c_i \,=\, \sum_{\gamma \in \GG_v} \ \sum_{i : \gamma \in U_i} c_i \,=\, \sum_{\gamma \in \GG_v } f(\gamma) \,=\, s_*f(v).
\]
It follows that $r_*f(x) = s_*f(x)$ for all $x \in \GGo$, and so $\delta_1(f) = 0$.

We now prove part~\cref{item: r_* and s_* of ample rep image}. Routine calculations show that for all $B \in \BcoG$, we have $r_*(1_B) = 1_{r(B)}$ and $s_*(1_B) = 1_{s(B)}$. Thus, for all $B \in F(\GG)$, we have $r_*(1_B) = 1_{\GGo} = s_*(1_B)$. Since $r_*$, $s_*$, and $\pi$ are all linear maps, it follows that
\[
r_*\big(\pi(\C F(\GG))\big) = \C 1_{\GGo} = s_*\big(\pi(\C F(\GG))\big). \qedhere
\]
\end{proof}

In order to prove \cref{cor: surj iff group}, we first utilise \cref{prop: ample rep image}\cref{item: r_* and s_* of ample rep image} to prove the following result. We thank the anonymous referee for suggesting this simple proof.

\begin{prop} \label{prop: TFG and im(pi)}
Let $\GG$ be an ample Hausdorff groupoid with compact unit space $\GGo$. If $B$ is a nonempty compact open bisection of $\GG$ such that $1_B \in \pi(\C F(\GG))$, then $B \in F(\GG)$.
\end{prop}

\begin{proof}
Let $B$ is a nonempty compact open bisection of $\GG$ such that $1_B \in \pi(\C F(\GG))$. By \cref{prop: ample rep image}\cref{item: r_* and s_* of ample rep image}, we know that $1_{r(B)} = r_*(1_B)$ and $1_{s(B)} = s_*(1_B)$ are both nonzero elements of $\C 1_{\GGo}$. It follows that $r(B) = \GGo = s(B)$, and hence $B \in F(\GG)$.
\end{proof}

The following result is an immediate corollary of \cref{prop: TFG and im(pi)}, because $A(\GG)$ is the span of characteristic functions on compact open bisections of $\GG$.

\begin{cor} \label{cor: ample rep not surj}
Let $\GG$ be an ample Hausdorff groupoid with compact unit space $\GGo$. If there exists a nonempty compact open bisection $B$ of $\GG$ such that $B \notin F(\GG)$, then the representation $\pi\colon \C F(\GG) \to A(\GG)$ is not surjective.
\end{cor}

As the following corollary shows, it turns out that the hypothesis of \cref{cor: ample rep not surj} is very easily satisfied, as it holds whenever $\GG$ is not a group.

\begin{cor} \label{cor: surj iff group}
Let $\GG$ be an ample Hausdorff groupoid with compact unit space $\GGo$. The representation $\pi\colon \C F(\GG) \to A(\GG)$ is surjective if and only if $\GG$ is a group.
\end{cor}

\begin{proof}
If $\GG$ is a group, then $F(\GG) \cong \GG$, so $\C F(\GG) = A(\GG)$, and $\pi\colon \C F(\GG) \to A(\GG)$ is the identity map and hence is surjective. For the converse, suppose that $\GG$ is not a group, and fix distinct units $u, v \in \GGo$. Since $\GG$ is an ample Hausdorff groupoid, there exist disjoint compact open sets $U, V \subseteq \GGo$ containing $u$ and $v$, respectively. But then $v \notin U$, so $U \in \BcoG {\setminus} F(\GG)$, and hence \cref{cor: ample rep not surj} implies that $\pi\colon \C F(\GG) \to A(\GG)$ is not surjective.
\end{proof}

\section{Representations of topological full groups of discrete groupoids}
\label{sec: discrete groupoids}

In this section we restrict our attention to discrete groupoids, and to the images of the representations of their topological full groups in the full and reduced groupoid C*-algebras. In particular, we prove an analogue of \cref{cor: surj iff group} for the extension of the representation $\pi$ with respect to the full C*-norm (see \cref{thm: discrete full iff group}), and we show in \cref{eg: F_2 sqcup F_2} that \cref{thm: discrete full iff group} does not hold in the reduced setting. We conclude the section by connecting our results from \cref{sec: injectivity,sec: surjectivity,sec: discrete groupoids} in \cref{cor: isomorphism}.

Let $\GG$ be a discrete groupoid with finite unit space $\GGo = \{a_1, \dotsc, a_n\}$. Recall that, for a groupoid $\GG$ and $a, b \in \GGo$, we define $\GG_b^a \coloneqq \{ \gamma \in \GG \mid r(\gamma) = a \text{ and } s(\gamma) = b \}$. Thus $\PG \coloneqq \big\{ \GG_{a_j}^{a_i} : i, j \in \{1, \dotsc, n\} \big\}$ is a partition of $\GG$ into disjoint sets. For $\gamma \in \GG$, write $1_\gamma \coloneqq 1_{\{\gamma\}} \in A(\GG)$. Given $f \in A(\GG)$ and $i, j \in \{1, \dotsc, n\}$, we define a map $f_{i,j}\colon \GG \to \C$ by
\[
f_{i,j}(\gamma) \coloneqq \begin{cases}
f(\gamma) & \text{if } \gamma \in \GG_{a_j}^{a_i} \\
0 & \text{otherwise}.
\end{cases}
\]
Then each $f_{i,j} \in A(\GG)$, and since $\PG$ is a partition of $\GG$, it follows that $f = \displaystyle\sum_{i,j=1}^n f_{i,j}$.

Define $T\colon A(\GG) \to M_n(\C)$ by
\[
T(f)_{ij} \coloneqq \sum_{\gamma \in \GG_{a_j}^{a_i}} f(\gamma), \ \text{ for each } i, j \in \{1, \dotsc, n\}.
\]
We will use this map $T$ to study the image of the representation $\pi\colon \C F(\GG) \to A(\GG)$. We first show that $T$ is a $*$-representation of $A(\GG)$.

\begin{lemma} \label{lem: T *-rep}
Let $\GG$ be a discrete groupoid with finite unit space $\GGo \coloneqq \{a_1, \dotsc, a_n\}$. The map $T\colon A(\GG) \to M_n(\C)$ defined above is a $*$-representation of $A(\GG)$.
\end{lemma}

\begin{proof}
It is straightforward to verify that $T$ is linear. Fix $f, g \in A(\GG)$. For all $i, j \in \{1, \dotsc, n\}$, we have
\begin{align*}
T(f * g)_{ij} &= \sum_{\gamma \in \GG_{a_j}^{a_i}} (f * g)(\gamma) = \sum_{\gamma \in \GG_{a_j}^{a_i}} \sum_{\alpha\beta = \gamma} f(\alpha) g(\beta) \\
&= \sum_{k=1}^n \sum_{\alpha \in \GG_{a_k}^{a_i}} f(\alpha) \sum_{\beta \in \GG_{a_j}^{a_k}} g(\beta) = \sum_{k=1}^n T(f)_{ik} \, T(g)_{kj} = \big(T(f) T(g)\big)_{ij},
\end{align*}
and
\[
T(f^*)_{ij} = \sum_{\gamma \in \GG_{a_j}^{a_i}} f^*(\gamma) = \sum_{\gamma \in \GG_{a_j}^{a_i}} \overline{f(\gamma^{-1})} = \overline{\sum_{\eta \in \GG_{a_i}^{a_j}} f(\eta)} = \overline{T(f)_{ji}} = \big(T(f)^*\big)_{ij}.
\]
Thus $T(f * g) = T(f) T(g)$ and $T(f^*) = T(f)^*$, and so $T$ is a $*$-homomorphism.
\end{proof}

The following result is a corollary of \cref{prop: ample rep image}\cref{item: ample rep image in ker delta_1}.

\begin{cor} \label{cor: discrete image rep}
Let $\GG$ be a discrete groupoid with finite unit space $\GGo \coloneqq \{a_1, \dotsc, a_n\}$. Then
\[
\pi(\C F(\GG)) \subseteq \big\{ f \in A(\GG) \mid \exists \, c_f \in \C \text{ such that all row and column sums of } T(f) \text{ are } c_f \big\}.
\]
\end{cor}

\begin{proof}
Fix $f = \displaystyle\sum_{i,j=1}^n f_{i,j} \in \pi(\C F(\GG))$. Then, for each $i, j \in \{1, \dotsc, n\}$,
\[
\text{the $i$\textsuperscript{th} row sum of } T(f) = \sum_{k=1}^n T(f)_{ik} = \sum_{k=1}^n \sum_{\gamma \in \GG_{a_k}^{a_i}} f(\gamma) = \sum_{\gamma \in \GG^{a_i}} f(\gamma) = r_*f(a_i),
\]
and
\[
\text{the $j$\textsuperscript{th} column sum of } T(f) = \sum_{k=1}^n T(f)_{kj} = \sum_{k=1}^n \sum_{\gamma \in \GG_{a_j}^{a_k}} f(\gamma) = \sum_{\gamma \in \GG_{a_j}} f(\gamma) = s_*f(a_j).
\]
By \cref{prop: ample rep image}\cref{item: ample rep image in ker delta_1}, it follows that for all $i, j \in \{1, \dotsc, n\}$,
\[
\text{the $i$\textsuperscript{th} row sum of } T(f) \,=\, \text{the $j$\textsuperscript{th} column sum of } T(f). \qedhere
\]
\end{proof}

We now use \cref{cor: discrete image rep} to study the completions of $\pi(\C F(\GG))$ in the full and reduced groupoid C*-algebras. In \cref{thm: discrete full iff group} we prove that for a discrete groupoid $\GG$, an analogue of \cref{cor: surj iff group} holds for the full groupoid C*-algebra $C^*(\GG)$.

\begin{thm} \label{thm: discrete full iff group}
Let $\GG$ be a discrete groupoid with finite unit space $\GGo$. Then
\[
\overline{\pi(\C F(\GG))}^{\maxnorm{\cdot}} = C^*(\GG)
\]
if and only if $\GG$ is a group.
\end{thm}

\begin{proof}
If $\GG$ is a group, then $F(\GG) \cong \GG$, so $\C F(\GG) = A(\GG)$, and hence
\[
\overline{\pi(\C F(\GG))}^{\maxnorm{\cdot}} = \overline{A(\GG)}^{\maxnorm{\cdot}} = C^*(\GG).
\]

Suppose that $\GG$ is not a group. We show that $\overline{\pi(\C F(\GG))}^{\maxnorm{\cdot}} \ne C^*(\GG)$ by proving an even stronger result: that $1_\gamma \notin \overline{\pi(\C F(\GG))}^{\maxnorm{\cdot}}$ for each $\gamma \in \GG$. Write $\GGo = \{a_1, \dotsc, a_n\}$, and note that $n \ge 2$ since $\GG$ is not a group. Fix $\gamma \in \GG$, and suppose for contradiction that $1_\gamma \in \overline{\pi(\C F(\GG))}^{\maxnorm{\cdot}}$. Then there exists a sequence $(\varphi_m)_{m=0}^\infty$ of functions in $\pi(\C F(\GG))$ such that $\maxnorm{\varphi_m - 1_\gamma} \to 0$ as $m \to \infty$. By \cref{lem: T *-rep}, $T\colon A(\GG) \to M_n(\C)$ is a $*$-representation of $A(\GG)$, and hence \cref{eqn: Steinberg *-reps are bounded} on \cpageref{eqn: Steinberg *-reps are bounded} implies that $T$ is bounded. Thus
\begin{equation} \label{eqn: convergence of T(phi_m)}
\norm{T(\varphi_m) - T(1_\gamma)}_{M_n(\C)} \,=\, \norm{T(\varphi_m - 1_\gamma)}_{M_n(\C)} \,\to\, 0 \quad \text{as } m \to \infty.
\end{equation}
Let $\ell$ and $k$ be the unique elements of $\{1, \dotsc, n\}$ such that $\gamma \in \GG_{a_k}^{a_\ell}$. Note that each $T(\varphi_m)$ has $n \ge 2$ rows, and it follows from \cref{eqn: convergence of T(phi_m)} that for each $i \in \{1, \dotsc, n\}$, we have
\[
i^\textsuperscript{th} \text{ row sum of } T(\varphi_m) \,\to\, i^\textsuperscript{th} \text{ row sum of } T(1_\gamma) \,=\, \sum_{j=1}^n T(1_\gamma)_{ij} \,=\, T(1_\gamma)_{ik} \,=\,
\begin{cases}
1 & \text{if } i = \ell \\
0 & \text{otherwise}
\end{cases}
\]
as $m \to \infty$. But this contradicts \cref{cor: discrete image rep}, which says that for each $m \in \N$, all of the row (and column) sums of $T(\varphi_m)$ are equal. So we must have $1_\gamma \notin \overline{\pi(\C F(\GG))}^{\maxnorm{\cdot}}$.
\end{proof}

\begin{remark} \label{rem: Cuntz groupoid}
It is known that \cref{thm: discrete full iff group} does not hold for ample Hausdorff groupoids in general. For example, if $\GG$ is the Cuntz groupoid (that is, the boundary-path groupoid of the directed graph with a single vertex and two edges), then $F(\GG)$ is Thompson's group $V_2$, and the representation $\pi\colon \C(F(\GG)) \to A(\GG)$ extends to a surjective representation of $F(\GG)$ in the Cuntz algebra $\OO_2$; see \cite[Remark~4.7]{BS2019} and \cite[Proposition~5.3]{HO2017}.
\end{remark}

It turns out that \cref{thm: discrete full iff group} does not hold in the reduced setting. We provide an example demonstrating this fact below.

\begin{example} \label{eg: F_2 sqcup F_2}
Let $\GG = \F_2 \sqcup \F_2$. Then each element of $\GG$ is of the form $(g,k)$, where $g \in \F_2$, and $k \in \{1, 2\}$ identifies whether $g$ belongs to the first or the second copy of $\F_2$. Since $\GG$ is not a group, we know by \cref{thm: discrete full iff group} that $\overline{\pi(\C F(\GG))}^{\maxnorm{\cdot}} \ne C^*(\GG)$. We show that despite this, we still have $\overline{\pi(\C F(\GG))}^{\rnorm{\cdot}} = C_r^*(\GG)$. To do so, it suffices to show that for each $g \in \F_2$, we have $1_{(g,1)} \in \overline{\pi(\C F(\GG))}^{\rnorm{\cdot}}$, because a symmetric argument then shows that $1_{(g,2)} \in \overline{\pi(\C F(\GG))}^{\rnorm{\cdot}}$. Fix $t \in \F_2$, and for each $m \in \N$, let $E_m$ denote the set of (reduced) elements of $\F_2$ with length $m$. List elements of $\F_2$ in increasing order of their lengths; that is, write $\F_2 = \{g_1, g_2, g_3, \dotsc \}$, with $\abs{g_i} \le \abs{g_{i+1}}$ for all $i \ge 1$. Now define a sequence of functions $(\phi_n)_{n=1}^\infty \subseteq \pi(\C F(\GG)) \subseteq A(\GG)$ by
\[
\phi_n \coloneqq \pi\Big( \delta_{(t,1)} + \sum_{i=1}^n \tfrac{1}{n} \, \delta_{(g_i,2)} \Big) = 1_{(t,1)} + \frac{1}{n} \Big( \sum_{i=1}^n 1_{(g_i, 2)} \Big).
\]
We claim that $\phi_n \to 1_{(t,1)}$ in $C^*_r(\GG)$. Since the map $1_{g_i} \mapsto 1_{(g_i,2)}$ extends to an embedding of $C_r^*(\F_2)$ in $C_r^*(\GG)$, it suffices to show that $\psi_n \coloneqq \dfrac{1}{n} \Big( \displaystyle\sum_{i=1}^n 1_{g_i} \Big) \to 0$ in $C_r^*(\F_2)$.

By \cite[Lemma~1.5]{Haagerup1978}, we know that for all $f \in C_c(\F_2)$,
\begin{equation} \label[inequality]{ineq: Haagerup's inequality}
\rnorm{f} \,\le\, 2 \, \Big( \sum_{s \in \F_2} \abs{f(s)}^2 \, \big( 1 + \abs{s}^4 \big) \Big)^{\tfrac{1}{2}}.
\end{equation}
For each $m \ge 1$, we have $\abs{E_m} = 4 \times 3^{m-1}$. Thus, for each $n \ge 1$, we have
\[
\sum_{m=0}^{\ceil{\log_3{n}}} \abs{E_m} \,=\, \abs{E_0} + 4 \sum_{m=1}^{\ceil{\log_3{n}}} 3^{m-1} \,=\, 1 + \frac{4\big(3^{\ceil{\log_3{n}}} - 1\big)}{3 - 1} \,\ge\, 1 + \frac{4(n - 1)}{2} \,\ge\, n,
\]
and it follows that $\supp(\psi_n) = \{g_1, \dotsc, g_n\} \subseteq \bigcup_{m=0}^{\ceil{\log_3{n}}} E_m$.

Now, for each $n \ge 1$, \cref{ineq: Haagerup's inequality} implies that
\begin{align*}
\rnorm{\psi_n} \,&\le\, 2 \, \Big( \sum_{s \in \F_2} \abs{\psi_n(s)}^2 \, \big( 1 + \abs{s}^4 \big) \Big)^{\tfrac{1}{2}} \,=\, 2 \, \Big( \sum_{m=0}^{\ceil{\log_3{n}}} \sum_{s \in E_m} \abs{\psi_n(s)}^2 \, \big( 1 + \abs{s}^4 \big) \Big)^{\tfrac{1}{2}} \\
&\le\, 2 \, \Big( \sum_{m=0}^{\ceil{\log_3{n}}} \frac{\abs{E_m}}{n^2} \, (1 + m^4) \Big)^{\tfrac{1}{2}} \,\le\, 2 \, \Big( \sum_{m=0}^{\ceil{\log_3{n}}} \frac{4 \times 3^{m-1} \times 2m^4}{n^2} \Big)^{\tfrac{1}{2}} \\
&\le\, 2 \, \Big( \frac{8 \ceil{\log_3{n}}^4}{n^2} \sum_{m=0}^{\ceil{\log_3{n}}} 3^{m-1} \Big)^{\tfrac{1}{2}} \,=\, 2 \, \Big( \frac{8 \ceil{\log_3{n}}^4 \, \big( 3^{1 + \ceil{\log_3{n}}} - \tfrac{1}{3} \big)}{n^2 \, (3 - 1)} \Big)^{\tfrac{1}{2}} \\
&\le\, \frac{4 \ceil{\log_3{n}}^2}{n} \big( 3^{2 + \log_3{n}}\big)^{\tfrac{1}{2}} \,=\, \frac{12 \ceil{\log_3{n}}^2}{\sqrt{n}}.
\end{align*}
Since $\dfrac{12 \ceil{\log_3{n}}^2}{\sqrt{n}} \,\to\, 0$ as $n \to \infty$, we deduce that $\psi_n \to 0$ in $C_r^*(\F_2)$, as required.
\end{example}

We conclude the paper with a corollary of \cref{thm: main,cor: surj iff group,thm: discrete full iff group}.

\begin{cor} \label{cor: isomorphism}
Let $\GG$ be an ample Hausdorff groupoid with compact unit space $\GGo$. The representation $\pi\colon \C F(\GG) \to A(\GG)$ is an isomorphism if and only if $\GG$ is a group. Similarly, if $\GG$ is discrete, then the extension $\overline{\pi}_{\max}\colon C^*(F(\GG)) \to C^*(\GG)$ of $\pi$ is an isomorphism if and only if $\GG$ is a group.
\end{cor}

\begin{proof}
If $\GG$ is a group, then $\GG$ satisfies \cref{cond: all isotropy} of \cref{thm: main}, so \cref{thm: main,cor: surj iff group} together imply that $\pi\colon \C F(\GG) \to A(\GG)$ is an isomorphism. If $\GG$ is not a group, then \cref{cor: surj iff group} implies that $\pi$ is not an isomorphism. Now suppose that $\GG$ is discrete. Since $\pi\colon \C F(\GG) \to C^*(\GG)$ is a $*$-homomorphism, it extends uniquely to a $*$\nobreakdash-homomorphism $\overline{\pi}_{\max}\colon C^*(F(\GG)) \to C^*(\GG)$. If $\GG$ is a group, then $F(\GG) \cong \GG$, so the representation $\pi\colon \C F(\GG) \to A(\GG)$ is the identity map, and thus the extension $\overline{\pi}_{\max}$ is an isomorphism. If $\GG$ is not a group, then \cref{thm: discrete full iff group} implies that $\overline{\pi}_{\max}$ is not an isomorphism.
\end{proof}

\vspace{2ex}
\bibliographystyle{amsplain}
\makeatletter\renewcommand\@biblabel[1]{[#1]}\makeatother
\bibliography{references}

\providecommand{\bysame}{\leavevmode\hbox to3em{\hrulefill}\thinspace}
\providecommand{\MR}{\relax\ifhmode\unskip\space\fi MR }
\providecommand{\MRhref}[2]{%
  \href{http://www.ams.org/mathscinet-getitem?mr=#1}{#2}
}
\providecommand{\href}[2]{#2}
\begin{thebibliography}{10}

\bibitem{ABHS2017}
P.~Ara, J.~Bosa, R.~Hazrat, and A.~Sims, \emph{Reconstruction of graded
  groupoids from graded {Steinberg} algebras}, Forum Math. \textbf{29} (2017),
  1023--1037, \doi{10.1515/forum-2016-0072}.

\bibitem{BHM2022}
J.~Belk, J.~Hyde, and F.~Matucci, \emph{Embedding {$\Q$} into a finitely
  presented group}, Bull. Amer. Math. Soc. (N.S.) \textbf{59} (2022), 561--567,
  \doi{10.1090/bull/1762}.

\bibitem{BS2019}
K.A. Brix and E.~Scarparo, \emph{{C*-simplicity} and representations of
  topological full groups of groupoids}, J. Funct. Anal. \textbf{277} (2019),
  2981--2996, \doi{10.1016/j.jfa.2019.06.014}.

\bibitem{CFST2014}
L.O. Clark, C.~Farthing, A.~Sims, and M.~Tomforde, \emph{A groupoid
  generalisation of {Leavitt} path algebras}, Semigroup Forum \textbf{89}
  (2014), 501--517, \doi{10.1007/s00233-014-9594-z}.

\bibitem{CZ2022}
L.O. Clark and J.~Zimmerman, \emph{A steinberg algebra approach to {\'etale}
  groupoid {C*-algebras}}, preprint, 2022, \doi{10.48550/arXiv.2203.00179}.

\bibitem{Exel2008}
R.~Exel, \emph{Inverse semigroups and combinatorial {C*-algebras}}, Bull. Braz.
  Math. Soc. (N.S.) \textbf{39} (2008), 191--313,
  \doi{10.1007/s00574-008-0080-7}.

\bibitem{Exel2010}
\bysame, \emph{Reconstructing a totally disconnected groupoid from its ample
  semigroup}, Proc. Amer. Math. Soc. \textbf{138} (2010), 2991--3001,
  \doi{10.1090/s0002-9939-10-10346-3}.

\bibitem{GPS1999}
T.~Giordano, I.F. Putnam, and C.F. Skau, \emph{Full groups of {Cantor} minimal
  systems}, Israel J. Math. \textbf{111} (1999), 285--320,
  \doi{10.1007/bf02810689}.

\bibitem{Haagerup1978}
U.~Haagerup, \emph{An example of a nonnuclear {C*-algebra}, which has the
  metric approximation property}, Invent. Math. \textbf{50} (1978), 279--293,
  \doi{10.1007/bf01410082}.

\bibitem{HO2017}
U.~Haagerup and K.K. Olesen, \emph{Non-inner amenability of the {T}hompson
  groups {$T$} and {$V$}}, J. Funct. Anal. \textbf{272} (2017), 4838--4852,
  \doi{10.1016/j.jfa.2017.02.003}.

\bibitem{JM2013}
K.~Juschenko and N.~Monod, \emph{Cantor systems, piecewise translations and
  simple amenable groups}, Ann. of Math. (2) \textbf{178} (2013), 775--787,
  \doi{10.4007/annals.2013.178.2.7}.

\bibitem{JNdlS2016}
K.~Juschenko, V.~Nekrashevych, and {M. de la} Salle, \emph{Extensions of
  amenable groups by recurrent groupoids}, Invent. Math. \textbf{206} (2016),
  837--867, \doi{10.1007/s00222-016-0664-6}.

\bibitem{LV2020}
M.V. Lawson and A.~Vdovina, \emph{Higher dimensional generalizations of the
  {T}hompson groups}, Adv. Math. \textbf{369} (2020), 1--56,
  \doi{10.1016/j.aim.2020.107191}.

\bibitem{LBMB2018}
A.~{Le Boudec} and N.~Matte Bon, \emph{Subgroup dynamics and {C*}-simplicity of
  groups of homeomorphisms}, Ann. Sci. \'Ec. Norm. Sup\'er. (4) \textbf{51}
  (2018), 557--602, \doi{10.24033/asens.2361}.

\bibitem{MM2017}
K.~Matsumoto and H.~Matui, \emph{{Full groups of Cuntz--Krieger algebras and
  Higman--Thompson groups}}, Groups Geom. Dyn. \textbf{11} (2017), 499--531,
  \doi{10.4171/GGD/405}.

\bibitem{Matui2012}
H.~Matui, \emph{Homology and topological full groups of {\'e}tale groupoids on
  totally disconnected spaces}, Proc. Lond. Math. Soc. \textbf{104} (2012),
  27--56, \doi{10.1112/plms/pdr029}.

\bibitem{Matui2015}
\bysame, \emph{Topological full groups of one-sided shifts of finite type}, J.
  Reine Angew. Math. \textbf{705} (2015), 35--84,
  \doi{10.1515/crelle-2013-0041}.

\bibitem{Matui2016}
\bysame, \emph{{\'E}tale groupoids arising from products of shifts of finite
  type}, Adv. Math. \textbf{303} (2016), 502--548,
  \doi{10.1016/j.aim.2016.08.023}.

\bibitem{Matui2017}
\bysame, \emph{Topological full groups of \'etale groupoids}, {Operator
  algebras and applications---the Abel Symposium 2015}, Abel Symp., vol.~12,
  Springer, [Cham], 2017, \doi{10.1007/978-3-319-39286-8_10}, pp.~203--230.

\bibitem{Nek2018}
V.~Nekrashevych, \emph{Palindromic subshifts and simple periodic groups of
  intermediate growth}, Ann. of Math. (2) \textbf{187} (2018), 667--719,
  \doi{10.4007/annals.2018.187.3.2}.

\bibitem{Nek2019}
\bysame, \emph{Simple groups of dynamical origin}, Ergodic Theory Dynam.
  Systems \textbf{39} (2019), 707--732, \doi{10.1017/etds.2017.47}.

\bibitem{NO2019}
P.~Nyland and E.~Ortega, \emph{Topological full groups of ample groupoids with
  applications to graph algebras}, Internat. J. Math. \textbf{30} (2019),
  1--66, \doi{10.1142/S0129167X19500186}.

\bibitem{Renault1980}
J.~Renault, \emph{{A Groupoid Approach to C*-Algebras}}, Lecture Notes in
  Math., vol. 793, Springer, Berlin, 1980, \doi{10.1007/bfb0091072}.

\bibitem{Scarparo2023}
E.~Scarparo, \emph{A dichotomy for topological full groups}, Canad. Math. Bull.
  \textbf{66} (2023), 610--616,
  \doi{https://doi.org/10.4153/S000843952200056X}.

\bibitem{Sims2020}
A.~Sims, \emph{Hausdorff \'etale groupoids and their {C*-algebras}}, Operator
  algebras and dynamics: groupoids, crossed products, and Rokhlin dimension
  (F.~Perera, ed.), Advanced Courses in Mathematics, CRM Barcelona,
  Birkh\"auser/Springer, 2020, \doi{10.1007/978-3-030-39713-5}.

\bibitem{SWZ2019}
R.~Skipper, S.~Witzel, and M.~Zaremsky, \emph{Simple groups separated by
  finiteness properties}, Invent. Math. \textbf{215} (2019), 713--740,
  \doi{10.1007/s00222-018-0835-8}.

\bibitem{Steinberg2010}
B.~Steinberg, \emph{A groupoid approach to discrete inverse semigroup
  algebras}, Adv. Math. \textbf{223} (2010), 689--727,
  \doi{10.1016/j.aim.2009.09.001}.

\bibitem{Yang2022}
D.~Yang, \emph{{Higman--Thompson-like groups of higher rank graph
  C*-algebras}}, Bull. Lond. Math. Soc. \textbf{54} (2022), 1470--1486,
  \doi{10.1112/blms.12641}.

\end{thebibliography}
\vspace{-0.75ex}
\end{document}